\documentclass{article}

\title{Splitting matchings and the Ryser-Brualdi-Stein conjecture for multisets}
\author{Michael Anastos\footnote{Institute of Science and Technology Austria, Klosterneuburg, Austria. This project has received funding from the European Union’s Horizon 2020 research and innovation
programme under the Marie Sk\l{}odowska-Curie grant agreement No 101034413
\includegraphics[width=4.7mm, height=3mm]{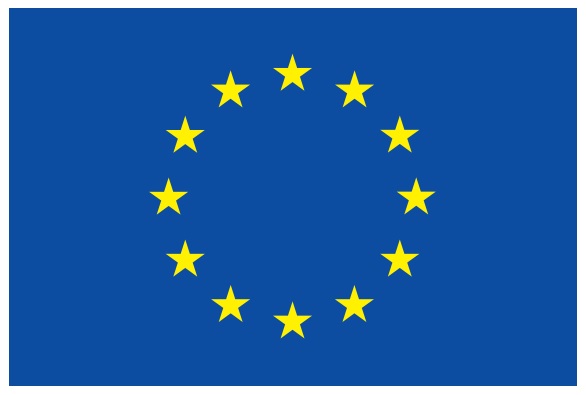}.} 
\and David Fabian\footnote{Freie Universität Berlin, Department of Mathematics and Computer Science, Berlin, Germany. Research supported by the Deutsche Forschungsgemeinschaft (DFG, German Research Foundation) Graduiertenkolleg “Facets of Complexity” (GRK 2434).} 
\and Alp M\"uyesser\footnote{University College London, London, UK.}  
\and Tibor Szabó\textsuperscript{\dag}}

\usepackage{graphicx}
\usepackage{amssymb}
\usepackage{amsthm}
\usepackage{amsmath}
\usepackage{enumitem}
\usepackage{framed}
\usepackage{soul}
\usepackage{changepage}
\usepackage{hyperref}
\usepackage[dvipsnames]{xcolor}
\usepackage{easyReview}

\theoremstyle{plain}
\newtheorem{theorem}{Theorem}[section]

\newtheorem{proposition}[theorem]{Proposition}

\newtheorem{conjecture}[theorem]{Conjecture}

\theoremstyle{definition}

\newcommand\N{\mathbb{N}}

\begin{document}
\maketitle
\begin{abstract} We study multigraphs whose edge-sets are the union of three perfect matchings, $M_1$, $M_2$, and $M_3$. Given such a graph $G$ and any $a_1,a_2,a_3\in \mathbb{N}$ with $a_1+a_2+a_3\leq n-2$, we show there exists a matching $M$ of $G$ with $|M\cap M_i|=a_i$ for each $i\in \{1,2,3\}$. The bound $n-2$ in the theorem is best possible in general.
We conjecture however that if $G$ is bipartite, the same result holds with $n-2$ replaced by $n-1$. We give a construction that shows such a result would be tight. We also make a conjecture generalising the Ryser-Brualdi-Stein conjecture with colour multiplicities.
\end{abstract}

\section{Introduction} 
Let $G$ be a graph on $2n$ vertices whose edge-set is the union of $k$ edge-disjoint perfect matchings. Alternatively, one can also imagine a properly $k$-edge-coloured $k$-regular graph, where the matchings are the colour classes. For which sequences $a_1,\ldots, a_k$ with $\sum_{i\in[k]}a_i\leq n$ does there exist a ``colourful'' matching $M$ of $G$ with the property that $|M\cap M_i|\geq a_i$ for each $i\in [k]$? This question was introduced by Arman, R\"odl, and Sales~\cite[Question 1.1]{arman2021colourful}. In their main result they obtained a couple of sufficient conditions for a relaxed version of the problem, where the base graph is $\ell$-regular and $\ell$-edge-coloured with a slightly larger $\ell \sim (1+\varepsilon)k$.  

In our paper we are mostly concerned with the original problem for three colours. Arguably, the first natural question is whether there exists a  ``fairly split'' perfect matching $M$, i.e. one with $|M\cap M_i| =n/3$ for every $i=1,2,3$. Of course $n$ has to be divisible by $3$ for this to have a chance of happening. 
It turns out that even if $3$ divides $n$, a fairly split perfect matching is only guaranteed to 
exist if $n=3$. Even more generally, for any $k\leq n-1$ or $k=n$ even, the only colour-multiplicity tuples $(a_1,\ldots, a_k)$ with $n=a_1 + \cdots + a_k$ which can be realised by a colorful perfect matching in any properly $k$-edge-coloured $k$-regular graph on $2n$ vertices are the trivial ones, namely those having a coordinate $n$. 

\begin{proposition}\label{prop:construction}
	Let $a_1, \ldots a_k \in \{ 0, 1, \ldots , n-1\}$ and  $n=a_1+\cdots +a_k$. For every $n>k$ or $n=k$ even, there exists a bipartite graph $G=(V,E)$ with $n$ vertices in each side whose edge set is the disjoint union of $k$ perfect matchings $M_1,\ldots ,M_k$, and there is no perfect matching $M$ of $G$ with $|M\cap M_i|=a_i$ for each $i\in[k]$.
\end{proposition}

The existence of a fairly split perfect matching for odd $k=n$ in bipartite graphs is known as Ryser's Conjecture, a famous and tantalising open problem.

As to the question of Arman, Rödl, and Sales for three colours, we show that a colourful matching of size as large as $n-2$ can always be found for any colour-multiplicity vector $(a_1,a_2,a_3)$. In fact, this can be guaranteed even when the matchings we start with are not necessarily disjoint.

\begin{theorem}\label{thm:arbitrary}
Let $G$ be a (multi-)graph on $2n$ vertices whose edge set is the disjoint union of three perfect matchings $M_1, M_2, M_3$. Then for any $a_1,a_2,a_3\in\N$ with $a_1+a_2+a_3 \leq n-2$ there exists a matching $M$ in $G$ such that $|M\cap M_1| = a_1$,  $|M\cap M_2| = a_2$, and $|M\cap M_3| = a_3$.
\end{theorem}

The proofs of the above theorem and Proposition \ref{prop:construction} are given in Section~\ref{sec:proof}. 

{\bf Remark 1.} In light of Proposition~\ref{prop:construction}, it is natural to ask how close to a fairly split perfect matching we can get for $k\geq 3$. 
Arman et al.~\cite{arman2021colourful} note that their results imply that one can always choose a matching $M$ with $|M\cap M_i| \geq n/k - \varepsilon n$ for every $i\in \{1,\ldots,k\}$. 
In their concluding remarks they also mention that their proof could be modified to establish the existence of a (smallest) constant $C_k$, depending only on $k$, such that a matching $M$ with $|M\cap M_i| \ge n/k - C_k$ for each $i\in\{1,\ldots,k\}$ can always be found. 
Proposition~\ref{prop:construction} shows that $C_k\geq 1$ for every $k$ and Theorem~\ref{thm:arbitrary} shows that $C_3=1$.
Using Alon's Necklace Theorem, as in \cite{arman2021colourful}, in combination with some extra combinatorial ideas, one can obtain a linear bound $C_k \leq 4k-6$ for all $k$. 
Since we believe that $C_k =1$ (cf Conjecture~\ref{con:optimistic}), we chose not to include the proof of that bound.

{\bf Remark 2.} We note that the bound $n-2$ in Theorem~\ref{thm:arbitrary} cannot be improved for general graphs without extra assumptions. To see this, for any even $n>2$ one can consider the (unique) decomposition of $n/2$ disjoint copies of $K_4$ into three perfect matchings $M_1,M_2,M_3$. 
Then the intersection of any matching $M$ of $G$ with any $K_4$ is a subset of some $M_i$, consequently the size of $M$ is at most $n$ minus the number of indices $i\in\{1,2,3\}$ for which $|M\cap M_i|$ is odd. 
Hence a matching $M$ of size $n-1=a_1+a_2+a_3$ with colour-multiplicity triple $(a_1,a_2,a_3)$ does not exist if $a_1,a_2,a_3$ are all odd. 

We conjecture that the construction from the previous remark is the only exception, i.e., a split with $a_1 + a_2 +a_3 = n-1$ should always possible if at least one component of $G$ is not a $K_4$. 

\begin{conjecture}\label{con:three-optimistic}
Let $G$ be a graph  on $2n$ vertices whose edge set is decomposed into perfect matchings $M_1, M_2$ and $M_3$ and let $a_1, a_2, a_3$ be non-negative integers such that $a_1+a_2+a_3=n-1$. If $G$ has a component that is not isomorphic to a $K_4$, then there exists a matching $M$ in $G$ such that $|M\cap M_i| = a_i$ for each $i\in\{1,2,3\}$.
\end{conjecture} 
A positive answer to this conjecture would in particular complete the resolution of the question of Arman et al. for three colours, as it implies that for a colour-multiplicity triple  $(a_1,a_2,a_3)$ with $a_1+a_2+a_3 = n-1$ a colourful matching is guaranteed to exist if and only if at least one of the $a_i$ is even. 
This would also imply that such a matching always exists if $n$ is odd.  

The construction in Proposition~\ref{prop:construction} is bipartite. We conjecture that the $n-2$ in Theorem~\ref{thm:arbitrary} can be replaced with $n-1$ if $G$ is assumed to be bipartite. (This is actually a special case of Conjecture~\ref{con:three-optimistic}.)
Even more generally, we suspect that for bipartite graphs the condition of Proposition~\ref{prop:construction} on the colour-multiplicities is best possible. More precisely, we conjecture that the following multiplicity version of the Ryser-Brualdi-Stein conjecture is true\footnote{Noga Alon independently also asked this as a question \cite{personalcommunication}.}.

\begin{conjecture}\label{con:optimistic}
Let $G$ be a complete bipartite graph on $2n$ vertices whose edge set is decomposed into perfect matchings $M_i$, $i=1,\ldots, n$. Let $a_i$, $i\in \{1,\ldots, n\}$ be a sequence of non-negative integers such that $\sum_i a_i=n-1$. Then, there exists a matching $M$ in $G$ such that $|M\cap M_i| = a_i$ for each $i\in\{1,\ldots, n\}$.
\end{conjecture}

Note that by K\"onig's Theorem any collection  of $k$ pairwise disjoint perfect matchings of $K_{n,n}$ can be extended to a collection of $n$ pairwise disjoint perfect matchings.
Therefore, if $G$ is bipartite the question of Arman et al for the colour multiplicity-tuple $(a_1,\ldots,a_k)$ is equivalent to the same question for the $n$-tuple $(a_1,\ldots,a_k,0,\ldots,0)$.
Conjecture \ref{con:optimistic} is easy to show when there are at most two non-zero colour-multiplicities.
The case of three non-zero colour-multiplicities, that is the strengthening of Theorem~\ref{thm:arbitrary} for bipartite graphs, is already open.
As in Theorem~\ref{thm:arbitrary}, Conjecture~\ref{con:optimistic} could also be true for multigraphs, but for simplicity we restrict ourselves to  simple graphs. 

Conjecture~\ref{con:optimistic} is quite optimistic, as it implies the Ryser-Brualdi-Stein conjecture (see \cite{keevash2022new} and the citations therein) by setting $a_i=1$ for all $i\in \{1,\ldots, n-1\}$ and $a_n=0$.
In fact, Conjecture~\ref{con:optimistic} is also related to the stronger Aharoni-Berger conjecture (see \cite{pokrovskiy2018approximate}).
Several other related generalisations of the Ryser-Brualdi-Stein conjecture have been previously proposed.
See for example Conjecture 1.9 in \cite{aharoni2017fair}, see also \cite{black2020fair}. 

{\bf Remark 3.} An old result of Hall \cite{hall} which was independently discovered by Salzborn and Szekeres \cite{salzborn} (see also \cite{ullman2019differences} for a modern exposition) shows that there can be no counterexample to Conjecture~\ref{con:optimistic} coming from addition tables of abelian groups (as in the proof of Proposition~\ref{prop:construction}).
It seems to be a problem of independent interest to generalise such results to non-abelian groups, which would give further evidence for Conjecture~\ref{con:optimistic}.

\section{Proofs}\label{sec:proof}
\begin{proof}[Proof of Proposition~\ref{prop:construction}.]
First we show that if $k<n$ or $k=n$ is even then there exist pairwise distinct $x_1, x_2, \ldots$ or $x_k \in \mathbb{Z}_n$ such that $a_1x_1+\cdots +a_kx_k\not\equiv 0 \pmod n$.  If $\sum_{i=1}^{k} i a_{\pi(i)} \not\equiv 0 \pmod n$ for some $\pi\in S_k$, then the choice $x_{\pi(i)}=i$ 
for every $i\in [k]$ works. This is certainly the case unless $a_1=\cdots = a_k=n/k$. In that case, if $n=k$ is even, then $\sum_{i=1}^n i \cdot 1\equiv n/2 \not\equiv 0 \pmod n$. If $n>k$ then, since none of the colour-multiplicities is $n$, we can assume without loss of generality that $a_k\not\equiv 0 \pmod n$. Then the choice $x_k = k+1$ and $x_i=i$ for every $i<k$ works, as then $\sum_{i=1}^k x_ia_i \equiv 0 + a_k \not\equiv 0 \pmod n$. Here note that since $k$ divides $n$ and $k<n$ we have $k\leq n/2$, so $k+ 1 <n$.

Let $G$ be a bipartite graph between two copies of the cyclic group $\mathbb{Z}_{n}$ consisting of the edges whose endpoints sum to $x_1, x_2, \ldots,$ or $x_k$. The edges whose endpoints sum to $x_i$ form a perfect matching $M_i$, and these matchings are pairwise disjoint. 
Suppose there exists a perfect matching $M$ of $G$ with $|M\cap M_i|=a_i$ for each $i\in [k]$. Summing up the endpoints of $M$ in two different ways, we obtain
$$a_1\cdot x_1+a_2\cdot x_2+\cdots + a_k\cdot x_k=\sum _{i\in \mathbb{Z}_{n}} i+ \sum_{i\in \mathbb{Z}_{n}}i.$$
\par Observe that the right hand side of the above equality is $0$ (for example, by pairing up inverses), which contradicts the choice of $x_1, x_2, \cdots , x_k$.
\end{proof}

\begin{proof}[Proof of Theorem~\ref{thm:arbitrary}.]
We say that a matching $M\subset E(G)$ is \emph{distributed as} $(a_1,a_2,a_3)$ if it satisfies $|M\cap M_1| = a_1$,  $|M\cap M_2| = a_2$, and $|M\cap M_3| = a_3$. It suffices to prove the claim for triples $(a_1,a_2,a_3)$ with $a_1 = \max\{ a_1,a_2,a_3 \}$ as the roles of the matchings are interchangeable.
We will show that given an $M$ that is distributed as $(a_1,a_2,a_3)$ with $a_1+a_2+a_3 = n-2$ we can find a matching $M'$ that is distributed as $(a_1-1,a_2+1,a_3)$.
This also implies the existence of matching distributed as $(a_1-1,a_2,a_3+1)$.
Starting from $M_1$ minus two arbitrary edges we can then find a matching distributed as $(a_1,a_2,a_3)$ for any such triple satisfying $a_1+a_2+a_3=n-2$.
\par For any matching $M \subset E(G)$ of size $n-2$ and any vertex $x$ that is unmatched by $M$, let $P_{23}(M,x)$ be the maximum $(M_2\setminus M)$-$(M_3\cap M)$-alternating path starting at $x$, and let $\ell_{23}(M,x)$ be its length. Let
    \[
    \ell_{23}(M) := \min_{x \text{ unmatched by } M} \ell_{23}(M,x).
    \]
For a matching $M$ of $G$ and $v\in V(G)$ denote by $M(v)$ the vertex $u$ that is matched by $M$ to $v$ i.e. $M(v)=u$ if and only if $\{v,u\}\in M$. 
Choose $M$ such that $\ell_{23}(M)$ is minimised over all matchings that are distributed as $(a_1,a_2,a_3)$.
Pick an unmatched vertex $x$ with $\ell_{23}(M,x) = \ell_{23}(M)$ and an unmatched vertex $z$ that is distinct from the endpoints of $P_{23}(M,x)$ and from $M_3(x)$.
We can choose such vertices because there are four unmatched vertices in total.
If $M_2(x)$ is incident to an edge of $M\cap M_1$ or unmatched we are done since in the former case the matching
    \[
    M \setminus \{ M_2(x)M_1(M_2(x)) \} \cup \{ xM_2(x) \}
    \] 
is distributed as $(a_1-1,a_2+1,a_3)$ while in the latter we can pick
    \[
    M \setminus \{ e \} \cup \{ xM_2(x) \}
    \] 
for any $e\in M\cap M_1$.
Hence we assume that $M_2(x)$ is incident to an edge of $M\cap M_3$.
Now $M_3(z)$ cannot be incident to an edge of $M\cap M_2$ because
    \[
    M' := M \setminus \{ M_2(x)M_3(M_2(x)), M_3(z)M_2(M_3(z)) \} \cup \{ xM_2(x), zM_3(z) \}
    \]
would be a matching that is distributed as $(a_1,a_2,a_3)$ and in which $P_{23}(M', M_3(M_2(x)))$ would be a path of length $\ell_{23}(M,x) -2$, which contradicts our choice of $M$.
Here it was important that $z$ is different from the endpoints of $P_{23}(M,x)$ so $P_{23}(M', M_3(M_2(x)))$ is a subpath of $P_{23}(M,x)$ not containing $x$ and therefore $P_{23}(M', M_3(M_2(x)))$ has smaller length than $P_{23}(M,x)$.
Therefore $M_3(z)$ is unmatched or incident to an edge of $M\cap M_1$.
If $M_3(z)$ is incident to $M\cap M_1$ then
    \[
    M'':= M \setminus \{ M_2(x)M_3(M_2(x)) , M_3(z)M_1(M_3(z))\} \cup \{ xM_2(x), zM_3(z) \}
    \]
is the desired matching.
Should $M_3(z)$ be unmatched then for any $e\in M \cap M_1$,
    \[
    M''':=	M \setminus \{ M_2(x)M_3(M_2(x)) , e\} \cup \{ xM_2(x), zM_3(z) \}
    \]
is distributed as $(a_1-1,a_2+1,a_3)$. Here we used that $M_3(x)\neq z$, or equivalently that $M_3(z)\neq x$. So under the previous that assumption $M_2(x)$ is incident to an edge in $M\cap M_3$, we have that the edges $xM_2(x), zM_3(z)$ are disjoint. Hence $M''$ and $M'''$ are indeed matchings of $G$.
\end{proof}

\bibliographystyle{plain}
\bibliography{bib}

\end{document}